\documentclass{amsart}
\usepackage{amssymb,amsmath,amsrefs}
\usepackage[colorlinks=true]{hyperref}
\hypersetup{urlcolor=blue, citecolor=green}
\usepackage[dvips]{graphicx}
\usepackage{color}

\usepackage{float}
\theoremstyle{plain}

\newtheorem{mainthm}{Theorem}

\newtheorem{mainclly}[mainthm]{Corollary}
\newtheorem{conj*}{Conjecture}
\newtheorem*{cor*}{Corollary}
\newtheorem{theorem}{Theorem}[section]

\newtheorem{lemma}[theorem]{Lemma}

\newtheorem{question}{Question}

\theoremstyle{definition}
\newtheorem*{def*}{Definition}

\newtheorem{rmk}[theorem]{Remark}
\newtheorem{example}[theorem]{Example}

\newtheorem{definition}[theorem]{Definition}

\newcommand{\be} {\beta}        
    
\newcommand{\de} {\delta}

\renewcommand{\epsilon}{\varepsilon}

\newcommand{\si} {\sigma}

\newcommand{\Z}{\mathbb{Z}}

\newcommand{\R}{\mathbb{R}}
\newcommand{\eps}{\varepsilon}

\title{On the Shadowableness of Flows With  Hyperbolic Singularities}

\author{Alexander Arbieto}
\address{Instituto de Matem\'{a}tica, Universidade Federal do Rio de Janeiro, Rio de Janeiro, Brazil}

\email{arbieto@im.ufrj.br}

\author{Andr\'{e}s M. L\'{o}pez}
\address{ Departamento de Matem\'{a}tica, Universidade Federal Rural do Rio de Janeiro, Serop\'{e}dica, Brazil.}

\email{andresmlopezb@gmail.com
}

\author{Elias Rego}

\address{Department of Mathematics, Southern University of science and Technlogy, Shenzhen, Gaungdong , China}

\email{rego@sustech.edu.cn}

\author{Yeison S\'{a}nchez}

\address{Departamento de matem\'{a}ticas, Universidad Nacional de Colombia, Bogot\'{a}, Colombia
}

\email{yasanchezr@unal.edu.co
}

\thanks{Primary MSC code: 37C10, secundary MSC code:37C50. Keywords: Flows, shadowing property, hyperbolic singularities, rescaled-shadowing property. 
The first author was partially supported by CNPq, FAPERJ and
PRONEX/DS from Brazil}

\begin{document}

\maketitle

\begin{abstract}
In this work we study the existence of singular flows satisfying shadowing-like properties. More precisely, we prove that if $C^1$-vector field on a closed manifold induces a chain-recurrent flow containing an attached hyperbolic singularity of stable or unstable index-one, then this flow cannot satisfy the shadowing property. If the manifold is non-compact, the vector field is complete and non-wandering, we prove that we prove that the existence of index-one hyperbolic singularities prevents the induced flow to satisfy the rescaled-shadowing property introduced in \cite{JNY}.

\end{abstract}

\section{Introduction}\label{intro}

The shadowing property, also called as the pseudo-orbit tracing property, is an important concept in the theory of dynamical systems. Roughly speaking, it says that any sequence of pieces of orbits (a pseudo-orbit), with sufficiently small jumps can be approximated by a real orbit (a shadow), see Section \ref{sectionstatements} for the precise definition. The shadowing property has an intrinsic relation with the stability of the dynamics, since the orbit of a nearby system is a pseudo-orbit of the original system. Thus, the shadows generated by these processes are the natural candidates to construct a conjugacy between the systems. As a consequence, the shadowing property has many interesting practical consequences in applied sciences. 	

The shadowing property became a subject of interest also due to its relations with the notion of hyperbolicity, introduced by S. Smale \cite{S}. The theory of hyperbolic dynamical systems is a landmark in the dynamical systems theory.  In the past decades hyperbolicity has proven to be a central subject  and a huge source of systems with rich dynamical behaviour.
 It was proved that hyperbolic systems have the shadowing property. More precisely, semi-locally, it was proved that hyperbolic basic sets have the shadowing property and,  globally, Axiom A diffeomorphisms with the strong transversality condition also have the shadowing property. As a consequence, such systems are structurally stable. Actually, K. Sakai showed that structural stable diffeomorphisms have the shadowing property, see \cite{Sak1} and \cite{Sak2}. Moreover, S. Pilyugin showed in \cite{Py1} that structurally stable diffeomorphisms actually satisfies a stronger form of shadowing, called the Lipschitz shadowing property. Reciprocally, even in the topological scenario, it was proved that an expansive homeomorphism with the shadowing property is topologically stable, a weaker form of stability, see \cite{AH}. For versions of these results for vector fields, we refer the reader to  \cite{Py2}.

However, there are other important dynamical systems which have robust dynamical properties but they are not structurally stable \cite{Wi}. The paradigmatic example is the Geometric Lorenz Attractor, introduced independently in \cite{ABS} and \cite{Gu1}, which is a robustly transitive attracting set, with many interesting properties, even topological: they are expansive \cite{K1}, or statistical: they have physical and SRB measures \cite{APPV}, among others. Theses properties were studied in a more general context:  the sectional hyperbolic systems, introduced by Metzger and  Morales in \cite{Mo3}; and there is a well developed theory on them nowadays.

Since such systems are not stable, we could suspect that they also do not have the shadowing property. Indeed, it was proved in \cite{K2} by M. Komuro  that the Geometric Lorenz attractor does not has the shadowing property. This was extended in \cite{WW},  where L. Wen and X. Wen proved that if a  chain-transitive sectional hyperbolic set has the shadowing property, then they must be hyperbolic. We remark that the sectional hyperbolic structure on the tangent bundle is strongly used in their proof.

  When one deals with continuous-time systems, there is a typical phenomenon that represents an obstruction to hyperbolicity: the existence of attached singularities, i.e., singularities approximated by regular orbits. Actually, it is the main difference between the sectional hyperbolic and the hyperbolic theories.  Indeed, in the presence of such singularities, a dynamical system cannot be hyperbolic. So, in this paper we push further this philosophy, proving that the presence of such singularities is enough to rule out the shadowing property requiring only pointwise hyperbolicity at the singularities and not any tangent structure globally. In the next section we will precisely state our main results.

 \section{Statement of the Results}\label{sectionstatements} 
 
 In this section we will state our main results. First let us introduce the setting that we will work on. Throughout this work $M$  denotes a closed and smooth Riemannian manifold (unless otherwise stated). Let $d$ denote the metric on $M$ induced by its Riemannian metric. We denote by $X$ a $C^1$-vector field over $M$. Recall that $X$ induces $C^1$-flow on $M$ that will be denoted by $X_t$. We say that $\sigma\in M$ is a \textit{singularity} of $X$ if $X(\sigma)=0$. We say that $x\in M$ is a \textit{regular point} of $X$ if $x$ is not a singularity. A set  $\Lambda\subset M$ is said to be \textit{invariant} if $X_t(\Lambda)=\Lambda$, for every $t\in \R$.  We say that a singularity of $X$ is \textit{attached to $\Lambda$} if it is accumulated by regular points of $\Lambda$. 
 For $x\in M$ we define the orbit of $x$ to be the set $$O(x)=\{X_t(x);t\in \mathbb{R}\}.$$ 

 Next we define the main subject of this work. Let  $\de,T>0$.  We say that a sequence $(x_i,t_i)_{a}^b$, with $-\infty\leq a<b\leq \infty$, is a $\de$-$T$-pseudo orbit if $t_i\geq T$ and $$d(X_{t_i}(x_i),x_{i+1})\leq \de,$$ for every $a\leq i\leq b$. When  $a$ and $b$ are finite, we say that $(x_i,t_i)_{a}^b$ is a finite $\de$-$T$-pseudo orbit. Let us fix the following notation $$Rep=\left\{h\in\mathcal{H}^+(\mathbb{R}); h(0)=0\right\},$$ 

 where $\mathcal{H}^+(\mathbb{R})$ denotes the set of increasing homeomorphisms from $\R$ to $\R$.

A map belonging to the  previous set is called a reparametrization.  Given a $\de$-$T$-pseudo-orbit,  define the sequence $(s_i)$ inductively by $$s_i=\begin{cases}0, \textrm{ if } i=0\\
t_i+s_{i-1},\textrm{ if } i>0\\
-t_i+s_{i+1}, \textrm{ if } i<0
\end{cases}.$$   We say that a $\de$-$T$-pseudo orbit $(x_i,t_i)$ is \textit{$\eps$-shadowed} if there are $y\in M$ and $h\in Rep$ such that $$d(X_{h(t)}(y),X_{t}(x_i))\leq \eps,$$ for $t\in [s_i,s_{i+1})$. 

\begin{definition}[Shadowing property]
 Let  $\Lambda$ be a  closed and invariant subset of $M$. We say that $X$ has the  shadowing property on $\Lambda$ if for every $\eps>0$, there is some $\de>0$ such that every $\de$-$1$-pseudo-orbit $(x_i,t_i)_a^b$, satisfying $x_i\in \Lambda$ for every $a\leq i\leq b$, is $\eps$-shadowed by some point in $\Lambda$. We say that $X$ has the global shadowing property if $X$ has the shadowing property on $M$.
\end{definition}

 We say that a closed and invariant subset $\Lambda$ is chain-recurrent if for every $x\in \Lambda$ and every $\eps>0$, there is a finite $\eps$-$1$-pseudo orbit $(x_i,t_i)_{0}^N$ such that $x_0=x_N=x$. Let $\sigma$  be a singularity of  $X$. Recall that  the \textit{stable set}  and  the \textit{unstable  set} of $\sigma$ are respectively defined by 

 $$W^s(\sigma)=\{x\in M; d(X_t(x),\sigma)\to 0, t\to+\infty\}$$
  \begin{center}
     and
 \end{center}
$$W^u(\sigma)=\{x\in M; d(X_t(x),\sigma)\to 0, t\to-\infty\}$$ 

The next concept will be central in this work.
\begin{definition}
We say that the orbit $O(p)$ of a regular point $p$ is a homoclinic loop if there is a singularity $\sigma$ such that $p\in W^s(\sigma)\cap W^u(\sigma)$. In this case, we say that $O(p)$ is a homoclinic loop attached to $\sigma$. For a given singularity $\sigma$, let us denote by $HL(\sigma)$ the set of homoclinic loops attached to $\sigma$. 
\end{definition}

Essentially, a point in a homoclinic loop is born and dies in the same singularity.  As we will see on the next sections, the shadowing property together with some form recurrency nearby a hyperbolic singularity  allow us to create many distinct attached homoclinic loops. Another important concept that will be used is the concept of hyperbolic singularity which is defined as follows.   

\begin{definition}
A singularity $\sigma$ of $X$ is hyperbolic if there exist $C,\lambda>0$ such that 
\begin{enumerate}
\item $T_{\sigma}M= E^s\oplus E^u$, where $E^s$ and $E^u$ are $DX$-invariant, 
\item $||{DX_t}|_{E^s}|| \leq Ce^{-\lambda t}$ and $||{DX_{-t}}|_{E^u}|| \leq Ce^{-\lambda t}$, for every $t\geq 0$.
\end{enumerate}
\end{definition}

It is an immediate consequence of the celebrated Hartman-Grobman theorem  that hyperbolic singularities are dynamically isolated, i.e, there is a neighborhood $U$ of $\sigma$ such that $\sigma$ is the only singularity in $U$ and for any regular point $x\in U$ there is $t\in \R$ such that  $X_t(x)\notin U$ (see Section \ref{prel} for more details). As a consequence, if $\sigma$ is hyperbolic and $O(x)\in HL(\sigma)$, then one can consider very  neighborhoods $V$ of $\sigma$ such that $$\#\partial V\cap O(X)\leq \infty.$$ Here, $\partial V$ denotes the border of $V$. We call a neighborhood $V$ of $\sigma$ satisfying the previous condition an adapted neighborhood of $\sigma$.  
Thus, if $\sigma$ is a hyperbolic singularity attached to a invariant set $\Lambda$, then we can define for such an adapted $V$ of $\sigma$ in $\Lambda$ and  $O(x)\in HL(\sigma)$ the following number  $$n(x,V,\sigma)=\#\partial V\cap O(x).$$ 
The numbers $n(x,V,\sigma)$ will be important to our first main result which provides an obstruction for flows with hyperbolic singularities to have the shadowing property.  

\begin{mainthm}\label{fewhomocliniloops}
    Let $\Lambda$ be a  chain-recurrent set for a $C^1$-vector filed $X$.   Let $\sigma\in \Lambda$ be a hyperbolic singularity attached  to $\Lambda$. If for every adapted neighborhood $V$  of $\sigma$ there is $N>0$ such that $n(p,V,\sigma)\leq N$ for every $p\in HL(\sigma)$, then $\Lambda$ has not the shadowing  property.
\end{mainthm}
The idea behind the proof of this result the following: If  a chain-recurrent set contains an attached hyperbolic and has the shadowing property, then by only using the recurrence nearby $\sigma$, we can construct as many homoclinic loops as we want. This allow us to find a sequence of  points $x_k$ such that $n(x_k,V,\sigma)$ grows indefinitely, as $k$ goes to infinity, contradicting the hypothesis. We remark that almost all techniques used to obtain the previous theorem are of topological nature, the only exemption is that we explore the qualitative behavior of $X$ nearby $\sigma$. This behavior is crucial to our proof and it is ruled by the Hartman-Grobman Theorem, which requires the hyperbolicity of $\sigma$, but we totally drop any global tangent structure on $\Lambda$.   

 Despite the abstract character of Theorem \ref{fewhomocliniloops}, we can actually apply it to prove the non-shadowableness of concrete examples. To see this, let us recall that a hyperbolic singularity $\sigma$ is said to be an index-one singularity if either $\dim(E^s)=1$ or $\dim(E^u)=1$. Now we have the following result.

\begin{mainthm}\label{noshadowing}
Let $X$ be $C^1$-vector field on $M$. Suppose $\Lambda\subset M$ is a chain-recurrent set with an attached index-one hyperbolic singularity. Then $X$ has not the shadowing property on $\Lambda$.
\end{mainthm}

As a consequence of the previous result, we can totally solve the problem of the shadowableness of chain-recurrent sets with attached hyperbolic singularities for flows defined in low-dimensional manifolds. Precisely, we obtain the following corollary.

\begin{mainclly}\label{dim3}
Let $X$ be $C^1$-vector field on a closed Riemannian manifold $M$ such that $\dim(M)\leq 3$. If $\Lambda\subset M$ is a chain-recurrent set containing an attached hyperbolic singularity, then $X$   has not the shadowing property on $\Lambda$. 
\end{mainclly}

\begin{rmk} Although all the previous results are stated for subsets of $M$, we remark that they still hold globally. Precisely,  the same results hold if one replace $\Lambda$ by the whole manifold $M$. We just notice that  in the global case, the assumption that $\sigma$ is attached is superfluous
\end{rmk}

The next example shows that the chain-recurrence assumption on the previous results \ref{dim3} is crucial. 

\begin{example}
Suppose $M=S^2$ and let $X$ be a vector field generating the Morse-Smale flow on $S^2$ with exactly two singularities hyperbolic singularities $p$ and $q$ such that:

\begin{enumerate}
    \item The singularity $p$ is a hyperbolic sink and the singularity  $q$ is a hyperbolic source.
    \item  any regular point is attracted by $p$ in the future and attracted by $q$ in the past.  
\end{enumerate}

Consider  $\Lambda=M$. Notice that $\Lambda$ contains attached hyperbolic singularities but it is not a chain-recurrent set. In addition, $X$ has the shadowing property in $\Lambda$. Indeed, since $X$ is Morse-Smale, then by the results in \cite{Pe} $X$ is structurally-stable. In particular,  $X$ has the shadowing property on $M$.(see \cite{Py2}).   
\end{example}

In addition to the  previous results, we also investigate the shadowableness of flows with attached hyperbolic singularities which are defined on non-compact manifolds. As we will see trough Section \ref{sectionnoncompact}; the techniques developed to prove Theorems \ref{fewhomocliniloops}, \ref{noshadowing} Corollary \ref{dim3}; also hold when the phase space $M$ is not compact. So, all the previous results still hold in the non-compact scenario. However, when we move from the compact  to the non-compact setting, we need to be careful with the choice of shadowing property. Indeed, as showed in \cite{JNY}, the usual shadowing property defined to deal with flows on compact manifolds is not an invariant for equivalent flows, when the ambient manifold is not compact. In other words, it is not a dynamical property. This motivated the authors in \cite{JNY} to introduce a new version of shadowing property, the so called rescaled-shadowing property. The rescaled-shadowing emerged as an alternative to the usual shadowing property and it is, in fact, a dynamical property in the non-compact scenario.

Let $M$ be a Riemanninan manifold (compact or not) and let $X$ be a $C^1$-vector field. Since $M$ is not necessarily compact anymore, we cannot assure the existence   of a global flow $X_t$  induced by $X$ on  $M$. Thus, for the general case we will always assume that $X$ is a complete vector field over $M$. We postpone the precise definition of rescaled-shadowing and $R$-chain-recurrent set until Section \ref{sectionnoncompact}. Now we concentrate ourselves in to state the main result that we obtained regarding this new shadowing property.

\begin{mainthm}\label{norshadowing}
Let $X$ be a complete $C^1$-vector field on a Riemannian manifold $M$. Let $\Lambda$ be a $R$-chain-recurrent set.

\begin{enumerate}
    \item  If $\Lambda$ contains  a hyperbolic attached singularity $\sigma$ and  for every adapted neighborhood $V$ of $\sigma$ there is $N>0$ such that $n(p,V,\sigma)\leq N$ for every $p\in HL(\sigma)$, then $X$ has not the rescaled-shadowing property on $\Lambda$.

    \item If $\Lambda$ contains an attached codimension-one hyperbolic singularity, then $X$  has not the rescaled-shadowing property on $\Lambda$.
    \item If $\dim(M)\leq 3$ and $\Lambda$ contains an attached hyperbolic singularity, then $X$  has not the rescaled-shadowing property on $\Lambda$.
\end{enumerate}

\end{mainthm}
The main ideas used for proving the previous theorem are similar to the ideas used to obtain Theorems \ref{fewhomocliniloops}, \ref{noshadowing} and Corollary \ref{dim3}; but the proof of Theorem \ref{norshadowing} is not just a straightforward adaptation of the arguments for the compact case. Indeed, the proof here is more delicate. The crucial distinction between shadowing property and rescaled-shadowing consists on the nature of the pseudo-orbits. Indeed, in the compact case we will construct pseudo-orbits containing both singularities and regular points at the same time. On the other hand, when we are trying construct pseudo-orbits for the rescaled-shadowing, we are not allowed to jump from a regular point to a singularity and vice-versa.  To overcome this issue, in the non-compact case we performed new pseudo-orbit constructions, by using some ideas derived from  the techniques developed in \cite{SYY}.         

The remaining of this paper is organized as follows: In section \ref{prel} state and discuss some basic concepts and some known results that we will use to obtain our results. In Section \ref{compact case} we provide the proofs for the Theorems \ref{fewhomocliniloops} \ref{noshadowing} and Corollary \ref{dim3}. In Section \ref{sectionnoncompact} we discuss the differences between the compact and non-compact cases as well as we provide a proof for Theorem \ref{norshadowing}.

\section{Preliminaries}\label{prel}

This section is devoted to state some concepts and previously known results that will be used in this text. 
\vspace{0.1in}

\textit{Basic Setting.}
\vspace{0.1in}

We start by fixing the setting that we will mainly work on. Throughout this work $M$ denotes a closed Riemannian manifold (unless otherwise stated). We denote $X$ for a $C^1$-vector field on $M$. We denote $\mathcal{X}^1(M)$ for the set of $C^1$-vector fields over $M$ and $\mathcal{H}^+(\mathbb{R})$ for the set of increasing homeomorphisms from $\R$ to $\R$.

\begin{definition}
A  $C^1$-flow on a compact metric space $M$ is a  $C^1$-map $\phi: \R \times M\to M$ such that $\phi(0,x)=x$ and $\phi(t+s,x)=\phi(s,(\phi(t,x))$, for every $t,s\in \R$ and $x\in M$.
\end{definition}

We denote by $\phi_t=\phi(t,\cdot)$ the time $t$ map of $\phi$. Equivalently, a $C^1$-flow can be seen as a one-parameter family of $C^1$-diffeomorphims $\{\phi_t\}_{t\in\R}$ such that $\phi_0=Id_M$ and $\phi_t\circ\phi_s=\phi_{t+s}$, for every $s,t\in \R$. By abuse of notation we will often denote the family $\{\phi_t\}_{t\in \R}$ by $\phi_t$. Recall that every $C^1$-vector field $X$ over $M$ defines a $C^1$-flow on $M$ which here will be denoted by $X_t$. A subset $\Lambda$ of $M$ is said to be \textit{invariant} if $X_t(\Lambda)=\Lambda$, for every $t\in \R$. We say that $x\in M$ is  a \textit{non-wandering point} if for any neighborhood $U$ of $x$ and $T>0$, there are $y\in U$ and $t>T$ such that  $X_t(y)\in U$. We denote by $\Omega(X)$ the set of non-wandering points of $X$. \textit{$X$ is non-wandering}  if $M=\Omega(X)$. A closed invariant set $\Lambda$ is said to be a \textit{non-wandering set for $X$} if the restricted flow $X_t|_{\Lambda}$ is non-wandering. 
 
 Let $Sing(X)$ denote the set of singularities of $X$. Note that if $\sigma$ is a singularity if, and only if $X_t(\sigma)=\sigma$, for every $t\in \mathbb{R}$. If $\sigma$ is hyperbolic, we call $s=\dim(E^s)$ and $u=\dim(E^s)$ the stable index and unstable index of $\sigma$, respectively. We say that $\sigma$ is an \textit{index-one hyperbolic singularity} if $s=1$ or $u=1$. Hereafter, we will assume that all the index-one hyperbolic singularities that we will  consider here have unstable index equal to one. We would like to stress that there is no loss of  generality in such assumption. Indeed, all the constructions we will perform here can be reproduced when $s=1$.  
 We denote $|\cdot|$ and $d$ for the norm on $TM$ and the metric induced on $M$ by the Riemmanian metric of $M$, respectively.  We denote by $B_r(x)$ the open ball in $M$ centered in $x$ with radius $r>0$, for every $x\in M$ which is not a hyperbolic singularity. When $\sigma\in M$ is a hyperbolic singularity, $B_r(\sigma)$ will denote a different type of neighborhood of $\sigma$. To make it precise, note that if $\sigma$ is hyperbolic, then every $v\in T_{\sigma}M$ can be write as $v=v^s+v^u$, where $v^s\in E^s$ and $v^u\in E^u$. Now we denote 
 $$\mathcal{B}_{r}(\sigma)=\{v\in T_{\sigma} M; |v^s|, |v^u|\leq r  \}.$$
Let $\be_0>0$ be such that the exponential map $\exp_x:T_xM\to M$ is injective  in $B_{\be_0}(0)$, for every $x\in M$. The previous map is actually an isometry, if one endows $M$ and $T_xM$ with the $d$ and $|\cdot|$, respectively.  For every $0<r\leq\be_0$,  denote $B_r(\sigma)=\exp_{\sigma}(\mathcal{B}_{r}(\sigma))$. The neighborhood $B_{r}(\sigma)$ is a box neighborhood of $\sigma$.  Since box neighborhoods form a basis of neighborhoods for $\sigma$, then there is no difference between use box neighborhoods or use balls to deal with local problems around $\sigma$. In addition,  box neighborhoods are useful to visualize some aspects of the local dynamics of a flow near a hyperbolic singularity.  

 By the stable manifold theorem for hyperbolic singularities, for any hyperbolic singularity $\sigma$ the sets $W^s(\sigma)$ and $W^u(\sigma)$  are immersed submanifolds of $M$ tangent to $E^s_{\sigma}$ and $E^u_{\sigma}$, respectively. In addition, one has $\dim(E^s_{\sigma})=\dim W^s(\sigma)$ and  $\dim(E^u_{\sigma})=\dim W^u(\sigma)$. 

 Recall the concept of $\epsilon$-local stable and unstable sets:

 $$W^s_{\epsilon}(\sigma)=\{x\in M; d(X_t(x),\sigma)\leq \eps, t\geq 0\}$$
 \begin{center}
     and
 \end{center}
 $$W^u_{\eps}(\sigma)=\{x\in M; d(X_t(x),\sigma)\leq \eps, t\leq0\}.$$
By the stable manifold theorem, for $\eps>0$ small enough,  $W^s_{\eps}(\sigma)$ and $W^u_{\eps}(\sigma)$ are immersed submanifolds tangent to $E^s_{\sigma}$ and $E^u_{\sigma}$, respectively. In addition, one has 

$$W^s(\sigma)=\bigcup_{t\geq0} X_{-t}(W^s_{\epsilon}(\sigma))
\textrm{ and } W^s(\sigma)=\bigcup_{t\geq0} X_{t}(W^u_{\epsilon}(\sigma)).$$ 

 A regular point $x$ is said to be a homoclinic point for a hyperbolic singularity $\sigma$ if $x\in W^s(\sigma)\cap W^u(\sigma)$. Let us denote by $HL(\sigma)$ the set of homoclinic points for $\sigma$. We will often refer to the orbit of a point in $HL(\sigma)$ as a homoclinic loop for $\sigma$. A crucial tool used to obtain our results is  the celebrated Grobman-Hartman theorem that we  state in the following. 
 
 \begin{theorem}[Hartman-Grobman Theorem \cite{MP}]\label{HG}
 Let $X$ be a $C^1$ vector field over $M$ and suppose $\sigma$ is a hyperbolic singularity of $X$. Then there are a neighborhood $V$ of $\sigma$, a neighborhood $U$ of $0\in T_{\sigma}M$ and a homeomorphim $h:V\to U$ such that $h(X_t(x))=DX_{t}(h(x))$, for every $x\in V$ and $t\in \R$ such that $X_t(x)\in V$. Here $DX_t$ denotes the local flow generated by the derivative of $X$ in $\sigma$. 
 \end{theorem}

\begin{rmk}
    The Hartman-Grobman theorem still holds when the manifold $M$ is not compact. 
\end{rmk}
 
Any neighborhood of $\sigma$ where the conjugacy given by the Hartman-Grobman theorem holds will be called a Hartman-Grobman neighborhood of $\sigma$. For a given Hartman-Grobman neighborhood $V$, denote $W^s_{loc}(\sigma)$ and $W^u_{loc}(\sigma)$ by $h(E^s(\sigma)\cap U)$ and $h(E^u(\sigma)\cap U)$, where $h:V\to U$ is the conjugacy given by the Hartman-Grobman theorem. By the Hartman-Grobman theorem, we can assume that $W^s_{loc}(\sigma)=W^s_{\eps}(\sigma)$ and $W^s_{loc}(\sigma)=W^s_{\eps}(\sigma)$ for some $\eps>0$.  If $\sigma$ is a hyperbolic singularity, we say that $\sigma$ is a sink if $E^u=\{0\}$, a souce if $E^s=\{0\}$ and a saddle if both $E^s$ and $E^u$ are non-trivial. If $V$ is a Hartman-Grobman neighborhood of $\sigma$, then any regular point $x\in V$ must leave $V$ in positive or in negative time. So, for a regular point $x\in V$ we define the following numbers

  $$t^+_x=\inf\{t>0; X_t(x)\in \partial V\} \textrm{ and } t_x^-=\inf\{t>0; X_{-t}(x)\in \partial V\},$$
when they are defined. 

\begin{rmk}
    Another consequence of the Hartman-Grobman theorem is that for $r>0$ small enough, the box neighborhood $B_{r}(\sigma)$ is an adapted neighborhood of $\sigma$. 
\end{rmk}

\vspace{0.1in}

\textit{Shadowing vs Strong Shadowing.}
\vspace{0.1in}

There are several different versions of the shadowing property, among the best
known are those proposed by Komuro (see \cite{K2}), who proposes 5 different types of
shadowing taking into account whether the $\delta$-$T$-chain is finite or infinite and the
type of reparameterization used for $\eps$-shadowing the pseudo-orbits. 
Before to precise their definition, let us fix the following notation  $$ Rep_{\eps}=\left\{h\in Rep; \left|\frac{h(s)-h(t)}{s-t}-1\right|\leq \eps\right\}.$$  

We say that a $\de$-$T$-pseudo orbit $(x_i,t_i)$ is \textit{strongly $\eps$-shadowed} if there are  $y\in M$ and  $h\in Rep_{\eps}$ such that $$d(X_{h(t)}(y),X_{t}(x_i))\leq \eps,$$ for $t\in [s_i,s_{i+1})$. 

\begin{definition}[Strong Shadowing property]
We say that $X$ has the strong shadowing property on $\Lambda$  if for every $\eps>0$, there is some $\de>0$ such that every   $\de$-$1$-pseudo-orbit $(x_i,t_i)_a^b$ such that $x_i\in \Lambda$, for $a\leq i\leq b$ is strongly $\eps$-shadowed by some point in $\Lambda$. We say that $X$ has the global strong shadowing property if $X$ has the strong shadowing property on $M$.
  
\end{definition}

It is immediate to see that the strong shadowing property implies the shadowing property, but as it was shown in \cite{K2} by M. Komuro, the converse does not hold. Indeed, Komuro introduced the strong shadowing property in order to recover a classical feature of shadowing in the homeomorphism setting: The equivalence between shadowing and finite shadowing. To make this equivalence clear in the context of flows, we recall the concept of finite shadowing property.

\begin{definition}(Finite Shadowing Property)
    \begin{enumerate}
        
    \item We say that $X$ has the finite shadowing property on $\Lambda$ if for every $\eps>0$, there exists $\delta>0$ such that every  finite $\de$-$1$-pseudo-orbit $(x_i,t_i)_a^b$ such that $x_i\in \Lambda$, for $a\leq i\leq b$ is $\eps$-shadowed  by some point in $\Lambda$. We say that $X$ has the global finite shadowing property  if $X$ has the finite shadowing  shadowing property on $M$
    \item We say that $X$ has the strong finite shadowing property on $\Lambda$ if for every $\eps>0$, there exists $\delta>0$ such that every   finite $\de$-$1$-pseudo-orbit $(x_i,t_i)_a^b$ such that $x_i\in \Lambda$, for $a\leq i\leq b$ is strongly $\eps$-shadowed by some point in $\Lambda$. We say that $X$ has the global finite strong  shadowing property if $X$ has the finite shadowing  shadowing property on $M$. 

\end{enumerate}
\end{definition}
As it is showed in \cite{K2}, if $\Lambda$ has not attached singularities; then shadowing, strong shadowing, finite shadowing and finite strong shadowing are equivalent. On the other hand, if $\Lambda$ has attached singularities, then the only equivalence we can obtain is between  strong shadowing and strong finite shadowing. We will come back to this topic on Sections \ref{sectionnoncompact}.

\begin{rmk}
    An equivalent way to define the shadowing property is to require  that for every $T \neq 0$ and any $\epsilon>0$, there exists $\delta>0$ such that every $\delta$-$T$-pseudo-orbit is
$\epsilon$-shadowed (see \cite{K2}).
\end{rmk}

\section{Proof of Theorems A, B and Corollary C}\label{compact case}

This section is devoted to the proof of Theorem A. The proof we will present here is based on some lemmas. Before to start the proof, we first explain the main steps of the  proof. The main idea here is that the shadowing property together with some recurrence nearby a hyperbolic singularity $\sigma$ is enough to construct several distinct homoclinic loops. Once we can construct such loops, the next step is to obtain a contradiction by using the shadowing property to explode the numbers $n(p,V,\sigma)$. 

We now start our route to prove Theorem A. The first lemma we need  is a classical fact of flows theory. Here we will provide a proof for the reader's convenience. 
\begin{lemma}\label{nonwand}
Let $\Lambda$ be a chain-recurrent set for $X$ and suppose $\Lambda$ satisfies the shadowing property. Then $\Lambda$ is non-wandering.
\end{lemma}

\begin{proof}
Suppose that $\Lambda$ is chain-recurrent and satisfies the shadowing property. Let $x\in \Lambda$ and take $U$ be a neighborhood of $x$  in $\Lambda$ and $T>0$. Let $\eps>0$ be such that $B_{\eps}(x)\cap \Lambda\subset U$ .  Let $0 <\delta\leq \eps$ be given by the shadowing property of $X$ in  $\Lambda$. Let $(x'_i,s'_i)_{1}^k$ be a $\delta$-$1$-pseudo-orbit such that  $x'_0=x'_k=x$ and construct a new pseudo orbit $(x_i,t_i)_{i\in \mathbb{Z}}$ by setting $x_{i}=x'_r$ and    $t_{i}=t'_r$, if $i=qk+r$ with $r\in [0,k]$ and  $q\in \Z$. Then the shadowing property implies that there is some $z\in \Lambda$ and some $h\in Rep$ such that  $d(z,x),d(X_{h(qs_k)}(z),x)\leq\eps$, for every $q\in \Z$. Since $h$ is increasing, there is some $q$ such that $h(qs_k)>T$ and therefore $\Lambda$ is non-wandering. 
\end{proof}
By the previous Lemma, from now on we will assume that the sets under consideration are non-wandering set. In our next result, we obtain a criterion for the existence of homoclinic loops  for non-wandering sets containing attached hyperbolic singularities.

\begin{lemma}\label{HL}
Let $\Lambda$ be a non-wandering set for $X$ which satisfies the shadowing property. If $\Lambda$ contains an attached hyperbolic singularity, then $HL(\sigma)\neq\emptyset$
\end{lemma}

\begin{proof}
Suppose $\Lambda$  is a non-wandering set for $X$  and suppose it  satisfies the shadowing property. For the remainder of this proof, all the neighborhoods we are considering here are neighborhoods in the relative topology of $\Lambda$.  Let $\si\in Sing(X)$ be a  an attached hyperbolic singularity of $X$. Since $\sigma$ is hyperbolic, by the Hartman-Grobman theorem there exists $r_0>0$ such that the flow generated by $X$ is conjugated to the flow generated by $DX_{\sigma}$ in  $V=B_{r_1}(\sigma)$. In particular,  the following facts hold:
\begin{itemize}
    \item $\sigma$ is the only one singularity in $V$. 
    \item  $\sigma$ cannot be a sink neither a source. 
\end{itemize}

The last fact holds since  $\Omega(X|_{\Lambda})=\Lambda$ and $\sigma$ is attached to $\Lambda$. Therefore, for every neighborhood $U$ of $\sigma$ there exists  a point $p\in U\setminus (W_{loc}^s(\sigma)\cup W_{loc}^u(\sigma))$.
Fix $\eps>0 $ such that $B_{2\eps}(\sigma)\subset V$ and let $0<\delta<\eps$ be given by the shadowing property.  Fix a regular point $p\in B_{\delta}(\sigma)\setminus W^s_{loc}(\sigma)$. Let  $U\subset B_{\delta}(\sigma)\setminus W^s_{loc}(\sigma)$ be a neighborhood of $p$. Since $p$ is a non-wandering point, there exists $x_0\in U$ and $t_0>t^+(x_0)$ such that $X_{t_0}(x_0)\in U$. By possible shrinking the neighborhood $V$,  we can suppose $t_0\geq 1$, since $\sigma$ is attached. 

 Note that we have $d(X_{t_0}(x_0),\sigma)\leq \delta$.
 Now we construct the following $\delta$-$1$-R-pseudo-orbit:
 
 $$(x_i,t_i)=\begin{cases} (\sigma,1), i< 0\\
(x_0,t_0), i=0\\
(\sigma,1), i\geq1 \end{cases}.$$

The shadowing property implies that there is some point $z\in \Lambda$ and a reparametrization $h$ such  that $$d(X_{h(t)}(z),X_{t}(x_i))\leq \eps$$ for every $t\in [t_i,t_{i+1}]$. Note that $z$ is a regular point. Indeed, since $z$  $\eps$-shadows  the orbit arc $X_{[0,t_0]}(x_0)$, then  $z$ needs to leave the neighborhood $B_{r_1}(\sigma)$.     
We claim that $x\in W^s(\sigma)\cap W^u(\sigma)$. Indeed,  since $$d(X_{h(t+1+t_0)}(z),\sigma)\leq\eps, \forall t\geq, 0$$ then $h_{1+t_0}(z)$ never leaves $V$ in positive time. Since $z$ is a regular point, then $z\in W^s(\sigma)$. Analogously we conclude that $z\in W^u(\sigma)$. Therefore the claim holds and the proof is complete. 

\end{proof}

Now we are able to prove Theorem \ref{fewhomocliniloops}.

\begin{proof}[Proof of Theorem \ref{fewhomocliniloops}.]
Suppose $\Lambda$ satisfies the shadowing property. Then $\Lambda$ is a non-wandering set by Lemma \ref{nonwand}. Let $\sigma$ be a hyperbolic singularity attached to $\Lambda$ and satisfying the hypothesis of the Theorem \ref{fewhomocliniloops}. Let $V$ be a Hartman-Grobman neighborhood of $\sigma$. Let $N$ be such that $n(p,V\sigma)\leq N$, for every $p\in HL(\sigma)$. Let $\eps>0$ be such that $B_{2\eps}(\sigma)\subset V$ and and Let $0<\delta\leq \eps$ be given by the shadowing property. Let $x\in B_{\delta}(\sigma)$ be a regular point of $\Lambda$. Let $U\subset B_{\delta}(\sigma)$ be a neighborhood of $x$ in $\Lambda$ which does not contain $\sigma$. Since $\Lambda$ is non-wandering, there is some point $x_0\in U$ and $t_0>t^+(p)$ such that $X_{t_0}(x_0)\in U$ . Let us construct the following $\delta$-$1$-pseudo orbit.

$$(x_i,t_i)=\begin{cases} (\sigma,1), i< 0\\
(x_0,t_0), i=1,....N+1\\
(\sigma,1), i>N+1\\
 \end{cases}.$$ 
 
 Let $z$ be a point $\eps$-shadowing for the previous pseudo-orbit. As in the proof of Lemma \ref{HL}, we have $z\in HL(\sigma)$. But the choice of $\eps$ and the construction of $(x_i,t_i)$ implies that $n(z,V,\sigma)\geq 2N$ and this is contradiction. Therefore, $\Lambda$ does not have the shadowing property.

\end{proof}

Now we are ready to prove Theorem \ref{noshadowing}.

\begin{proof}[Proof of Theorem \ref{noshadowing}]
Let $X_t$ be a non-wandering flow  with a hyperbolic singularity $\sigma$. Assume that $\dim(E^u(\sigma))=1$. Suppose that $\Lambda$ satisfy the shadowing property. Then by Lemma \ref{HL}, we have $HL(p)\neq \emptyset$. Since $dim(E^u(\sigma))=1$ and by the invariance of $W^u(\sigma)$ we have that $$W^u(\sigma)=O(p)\cup \{\sigma\}\cup O(q),$$ where $p$ and $q$ are points belonging to the two distinct components of $W_{loc}^u(\sigma)\setminus \{\sigma\}$. Therefore, there are at most two homoclinic loops on $HL(p)$ and then the proof follows by applying Theorem \ref{fewhomocliniloops}.     
\end{proof}

\begin{proof}[Proof of Corollary \ref{dim3}]
Suppose $\dim(M)\leq 3$ and $\Lambda$ is is chain-recurrent set for $X$ containing an attached hyperbolic singularity $\sigma$. Since $\Lambda$ is non-wandering, $\sigma$ must be a saddle. Therefore $\sigma$ is an index-one singularity and therefore the result follows by applying Theorem \ref{noshadowing}.

\end{proof}

\section{The non-compact case}\label{sectionnoncompact}
In this section we will discuss some aspects of the shadowing property for flows defined on non-compact spaces.  Throughout  this section, we denote $M$ a Riemanninan manifold and $X$ a complete vector field over $M$. In \cite{JNY} it was introduced a new version of shadowing property for flows on non-compact spaces, the so called rescaled-shadowing property. Next we recall this concept. Let us denote the following set

$$C^0_X(M)=\{e:M\to \R; e \textrm{ is continuous, }e(x)\geq0 \textrm{ and } e^{-1}(0)=Sing(X)\}.$$

\begin{definition}
    Given $e\in C^0_X(M)$ we say that a sequence $(x_i,t_i)_{a}^{b}$ ($-\infty\leq a<b\leq \infty$), is  an $e$-$T$-pseudo-orbit if $t_i\geq T$ and $$d(X_{t_i}(x_i),x_{i+1})\leq e(x_i),$$ for every $a\leq i \leq b$.
\begin{enumerate}
    
\item  We say that a $d$-$T$-pseudo-orbit $(x_i,t_i)$ is \textit{$e$-shadowed} by $y\in M$ if there exists $h\in Rep$  such that $$d(X_{h(t)}(y),X_{t}(x_i))\leq e(X_{t}(x_i)),$$ for every $t\in [s_i,s_{i+1})$. 
\item We say that a $d$-$T$-pseudo-orbit $(x_i,t_i)$ is \textit{strongly $e$-shadowed} by $y\in M$ if there exists $h\in Rep_{e(x_0)}$  such that $$d(X_{h(t)}(y),X_{t}(x_i))\leq e(X_{t}(x_i)),$$ for every $t\in [s_i,s_{i+1})$.

\end{enumerate}

\end{definition}

\begin{definition}[Rescaled Shadowing property]
     Let $\Lambda$ be a closed and invariant subset of $M$.
     \begin{enumerate}
         \item  We say that $\Lambda$  has the rescaled-shadowing  property if for every $e\in C^0_X(M)$, there is some $d\in C^0_X(M)$ such that every $d$-$1$-pseudo-orbit in $\Lambda$ is $e$-shadowed  by some point in $\Lambda$. We say that $X$  has the global rescaled-shadowing property  if $M$ has the rescaled-shadowing property.
         \item  We say that $\Lambda$  has the strong rescaled-shadowing  property if for every $e\in C^0_X(M)$, there is some $d\in C^0_X(M)$ such that every $d$-$1$-pseudo-orbit in $\Lambda$ is strongly $e$-shadowed  by some point in $\Lambda$. We say that $X$  has the global strong rescaled-shadowing property  if $M$ has the global strong rescaled-shadowing property.
    \end{enumerate}
    \end{definition}

\begin{rmk}
The previous definition is actually weaker than the definition presented in \cite{JNY}. Indeed, in \cite{JNY}  the rescaled-shadowing property was introduced by requiring  strongly $e$-shadowing instead of $e$-shadowing. So, in our terminology the concept introduced in \cite{JNY} is called  strong rescaled-shadowing property. As we noticed in Section \ref{prel}, strong shadowing and shadowing are actually different concepts and therefore we would like to point out that this difference can also persist for the rescaled-shadowing property. Nevertheless, since the strong rescaled-shadowing property implies the rescaled-shadowing property, if we prove that some set does not satisfy the rescaled-shadowing property, then the same result will hold for the strong rescaled-shadowing property. 
\end{rmk}

We say that $x\in M$ is  \textit{$R$-chain-recurrent} if for every $e\in C_X^0(M)$ there is an $e$-$1$-pseudo orbit such that $(x_i,t_i)_{0}^N$ satisfying $x_0=x_N=x$.  A  closed and invariant set  $\Lambda\subset M$ is \textit{$R$-chain-recurrent} if every point $x\in \Lambda$ is $R$-chain-recurrent.

Now, remember that since we are working with complete $C^1$-vector fields on $M$, then $Sing(X)$ is a closed subset of $M$. In particular, it is possible to define a function $e\in C_X^0(M)$ such that $B_{e(x)}(x)\cap Sing(X)=\emptyset$, for every regular point $x$. Thus, if $(x_i,t_i)$ is an $e$-$T$-pseudo orbit and $x_i\in Sing(X)$, for some $i$, then $x_i\in Sing(X)$, for every $i$. This last observation represents an obstruction for reproducing the same techniques used to obtain Theorem \ref{fewhomocliniloops} in the rescaled-shadowing scenario. Indeed, if the function $d\in C_X^0(M)$ given by the rescaled-shadowing property is too small, then we cannot jump from a regular point to  a singularity. This makes impossible to construct homoclinic loops in the same way we did in Section \ref{compact case}. To overcome this obstacles we need to use a refinement of the Hartman-Grobman  theorem. The construction that we are going to perform here is inspired on the work \cite{SYY}, where the authors introduced a singular cross-section in a neighborhood of a hyperbolic singularity. Their results assume that $M$ is compact and give us a rich qualitative and quantitative portrait of the dynamics of the flow nearby a hyperbolic singularity. Here we are going to reproduce some of their ideas to the non-compact setting.

As we have seen in Section \ref{prel}, since $M$ is a Riemmanian manifold, for every $x\in M$, there is a number  $\eps_{x}>0$ such that $\exp_{x}: B_{\eps_{x}}(0)\to \exp_{x }(B_{\eps_{x}}(0))$ is a local diffeomorphism. Nevertheless, since we are not in the compact case, this number $\eps_x$ depends of $x$.  Since we will suppose that $\Lambda$ is non-wandering and $\sigma$ is attached to $\Lambda$, then from now on we are going to suppose that $\sigma$ is a saddle.  Let us denote $D_{r}(\si)=\exp_{\sigma}(\mathcal{D}_{\si}\cap B_{r}(0))$, where $$\mathcal{D}_{\si}=\{v\in T_{\si}M; |v^s|=|v^u|\}.$$ 
Fix $0<r_0<\eps_{\sigma}$ such that $B_r(\sigma)$ is a Hartman-Grobman neighborhood of $\sigma$, for every $0<r<r_0$. Fix $0<r<r_0$ and $K>0$ such that $e^{-n}<r$, for every $n\geq K$. Let us denote $$A_n=D_r(\sigma)\cap(B_{e^{-n}}(\sigma)\setminus B_{e^{-n-1}}(\sigma)).$$ 

Note that if $n\neq m$, then $A_n$ and $A_m$ are disjoint. The next lemma will be a key tool to our arguments.

\begin{lemma}
 If $x\in D_r(\sigma)$ is a regular point, then $X_{[-t^-_x,t^+_x]}(x)$ intersects $D_r(\sigma)$ only once.     
\end{lemma}
\begin{proof}
    This lemma is a simple consequence of the Hartman-Grobman theorem. Indeed,  since $B_r(\sigma)$ is a Hartman-Grobman neighborhood of $\sigma$, there is a neighborhood  $U$ of $0$ in $T_{\sigma}M$ such that if the dynamics of $X$ in $V$ is conjugated to the dynamics of $DX_{\sigma}$ in $U$. Now let $v\in U$ and write $v=v^s+v^u$. Since $\sigma$ is a saddle, we have that $|v^s|$ decreases and $|v^u|$ increases in the future, as well as $|v^s|$ increases and $|v^u|$ decreases in the past. Since for any $x\in D_r(\sigma)$ one has that $v=\exp_{\sigma}^{-1}(x)$  satisfies $|v^s|=|v^u|$, then the lemma holds.    
\end{proof}

Fix some $n\geq K$ and  denote $$V_n=\bigcup_{x\in A_n} X_{[-t_x^-,t^+_x]}(x).$$ 
The sets $V_n$ form bands that in their turn  form a partition of $V\setminus (W^s_{loc}(\sigma)\cup W^u_{loc}(\sigma)$ with the following property: The orbit of every point in the interior of $V_n$ must cross $D(\sigma)$ exactly in $A_n$ before to leave $V$ in positive or negative time. 

\begin{figure}[H]\label{bands}
     \centering
     \includegraphics[scale=0.6]{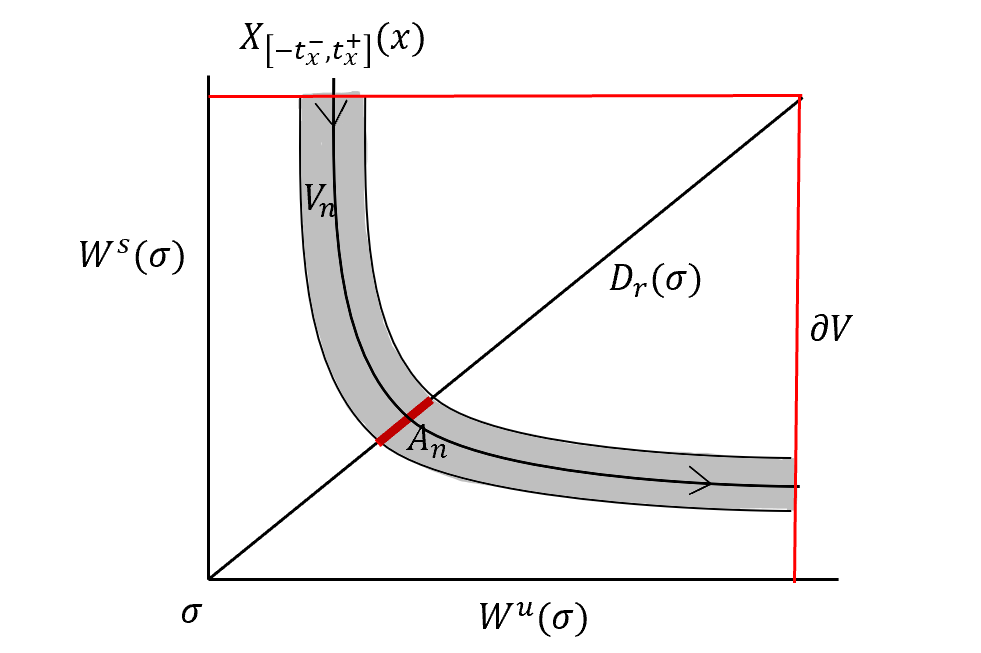}
     \caption{The Bands $V_n$.}
     \label{Bands2}
 \end{figure}

This idea will be central to the constructions that we are going to perform in the proofs of the next lemmas.

\begin{lemma}\label{Rnonwand}
Let $\Lambda$ be a R-chain-recurrent set for $X$ and suppose $\Lambda$ satisfies the  rescaled-shadowing property. Then $\Lambda$ is non-wandering.
\end{lemma}
\begin{proof}
The proof of this lemma is pretty similar to the proof  Lemma \ref{nonwand}. Suppose $\Lambda$ is a R-chain-recurrent set for $X$ and let $x\in \Lambda$. Fix $U$ a neighborhood of $x$ in $\Lambda$ and Let $e\in C_X^0(M)$ such that $(B_{2e(x)}(x)\cap \Lambda)\subset U$. Let $d\in C_X^0(M)$ be given by the rescaled-shadowing property and suppose $d(x)\leq e(x)$. Let $(x_i,t_i)_{0}^N$ be a $d$-$T$-pseudo orbit such that $x_0=x_N=x$. The shadowing property of $\Lambda$ implies there exists a point $z\in \Lambda$ such that $z,X_{h(s)}(z)\in U$.  
\end{proof}

\begin{lemma}\label{supoints}
Let $\sigma$ be a hyperbolic singularity for $X$ and suppose $p_n$ is a set of non-wandering points such that $p_n\to \sigma$. Suppose $V=B_{r}(\sigma)$, whith $0<r\leq r_0$. Then there there is a sequence of points $q_n\in D_r(\sigma)$ and positive times $u_n\leq t^+_{q_n}$ and $s_n> t^+_{q_n}$ such that $q_n\to \sigma$, $X_{u_n}(q_n)$ accumulates in $W^u_{loc}(\sigma)$ and  $X_{s_n}(q_n)$ accumulates in $W^s_{loc}(\sigma)$
\end{lemma}

\begin{proof}
To see why the lemma holds fix $V=B_{r}(\sigma)$, where $0<r\leq r_0$ and $B_{r}(\sigma)$. Let $p_n$ be a sequence of regular points such that $p_n\to \sigma$.  We can assume $p_n\in V$, for every $n\geq 0$.  For $n$ big enough we can assume that there is some $i_n\geq i_0$ such that $p_n\in V_{i_n}$. Without loss of generality, we can assume $p_n\in A_{i_n}$. Since $p_n$ is non-wandering, there exists some $q_n\in Int(A_{i_n-1}\cup A_{i_n}\cup A_{i_n+1})$ and $s'_n>t^+_{q_n}$ such that $X_{s'_n}(q_n)\in Int(A_{i_n-1}\cup A_{i_n}\cup A_{i_n
+1})$, where the interior is taken in the relative topology of $D_{\sigma}(r_0)$.  Denote $s_n=s'_n-t^-_{X_{s'_n}(q_n)}$ and $u_n=t^+(q_n)$.

Now, since $X$ is conjugated to $DX_{\sigma}$ in $V$, then $$X_{u_n}(q_n)\to W^u_{loc}(\sigma) \textrm{ and } X_{-t^-_{q_n}}(q_n)\to W^s_{loc}(\sigma),$$ see the Figure \ref{Accumulate}.
\begin{figure}[H]
     \centering
     \includegraphics[scale=0.6]{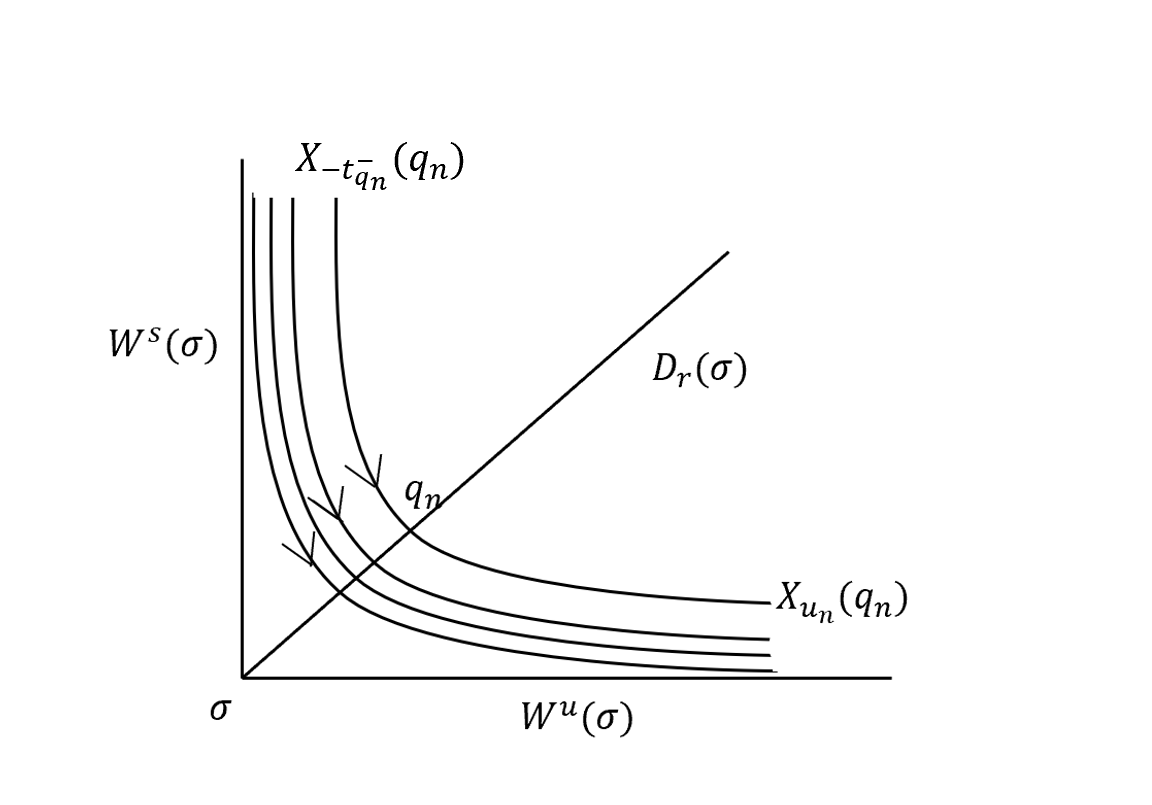}
     \caption{The sequences $X_{u_n}(q_n)$ and $X_{-t^-_{q_n}}(q_n)$ accumulating in $W^u_{loc}(\sigma)$ and $W^s_{loc}(\sigma)$, respectively.}
     \label{Accumulate}
 \end{figure}
 To conclude the proof, note that by construction,  as long as $n\to \infty$, one has $d(X_{-t^-(q_n)}(q_n),X_{s_n(q_n)}(q_n))\to 0$, since $diam((A_{i_n-1}\cup A_{i_n}\cup A_{i_n+1})\to 0$.

\end{proof}

\begin{lemma}\label{HLNC}
Let $\Lambda\subset M$ be a closed invariant set for a complete vector field $X$. Suppose that $\Lambda$ is non-wandering and  contains an attached hyperbolic singularity. If $X$  has the rescaled-shadowing property on $\Lambda$, then there exists a homoclinic loop in $HL(\sigma)$.

\end{lemma}

\begin{proof}

The idea behind the proof of this Lemma is pretty similar to the proof of Lemma \ref{HL}, but we need to be careful  with some details. As we have noticed before, since now we do not have control over the function $d$ given by  the rescaled-shadowing property, we cannot guarantee that we can jump from a regular point to a singularity as we  did in Lemma \ref{HL}. Thus we need to modify our argument. 

Suppose $\Lambda$  is a non-wandering set for $X$  and suppose it  satisfies the rescaled-shadowing property. For the remainder of this proof, all the neighborhoods we are considering here are neighborhoods in $\Lambda$.  Let $\si\in Sing(X)$ be an hyperbolic singularity attached to $\Lambda$. 

 Fix  $0<r_1<\frac{r_0}{2}$  and let $\mathcal{A}=\{A_n\}_{n\geq K}$ be the partition of $D_{\si}(r_1)$ as constructed above. Fix  $0<r_2\leq \frac{r_1}{2}$ such that  $$B_{r_2}(\partial B_{r_1}(\sigma))\cap B_{r_2}(\sigma)=\emptyset.$$  Fix $e\in C_X(M)$ such  that $e(x)<r_2$,  for every $x\in M$. Let $d\in C_X(M)$ be given by the rescaled-shadowing property and suppose $d(x)\leq e(x)$, for every $x\in M$.
 
By  Lemma \ref{supoints} There are  $x'_0\in W^u_{loc}(\sigma)$, $x_2\in W^s_{loc}(\sigma)$, $q_n\to \sigma$, $u_n>0$ and $s_n>0$ such that  $X_{u_n}(q_n)\to x'_0$ and $X_{s_n}(q_n)\to x_2$. Since $\Lambda$ closed then in particular $x'_0,x_2\in \Lambda$.  Furthermore, by the invariance of $\Lambda$, we can suppose that $x'_0,x_1\in \partial V$. Consider $x_0=X_{-1}(x_0')$ and
denote $$d_0=\inf\{d(x); x\in \partial V\}.$$
Note that $d_0>0$,  since $\partial V$ is compact and there are not singularities on $\partial V$. By the continuity of $X_t$, we can find an integer $n_0>K$ such that $q_{n_0}\in V$, $$d(X_{u_{n_0}}(q_{n_0}),x_2)\leq d_0   \textrm{ and } d(X_{s_{n_0}}(q_{n_0},x_0)\leq d_0.$$ Denote $x_1=q_{n_0}$ and consider the following sequence:

 $$(x_i,t_i)=\begin{cases} (X_{i}(x_0),1), i\leq 0\\
(x_1,s_{n_0}), i=1\\
(x_2,1), i\geq2\\
\end{cases}.$$ 

\begin{figure}[H]
     \centering
     \includegraphics[scale=0.3]{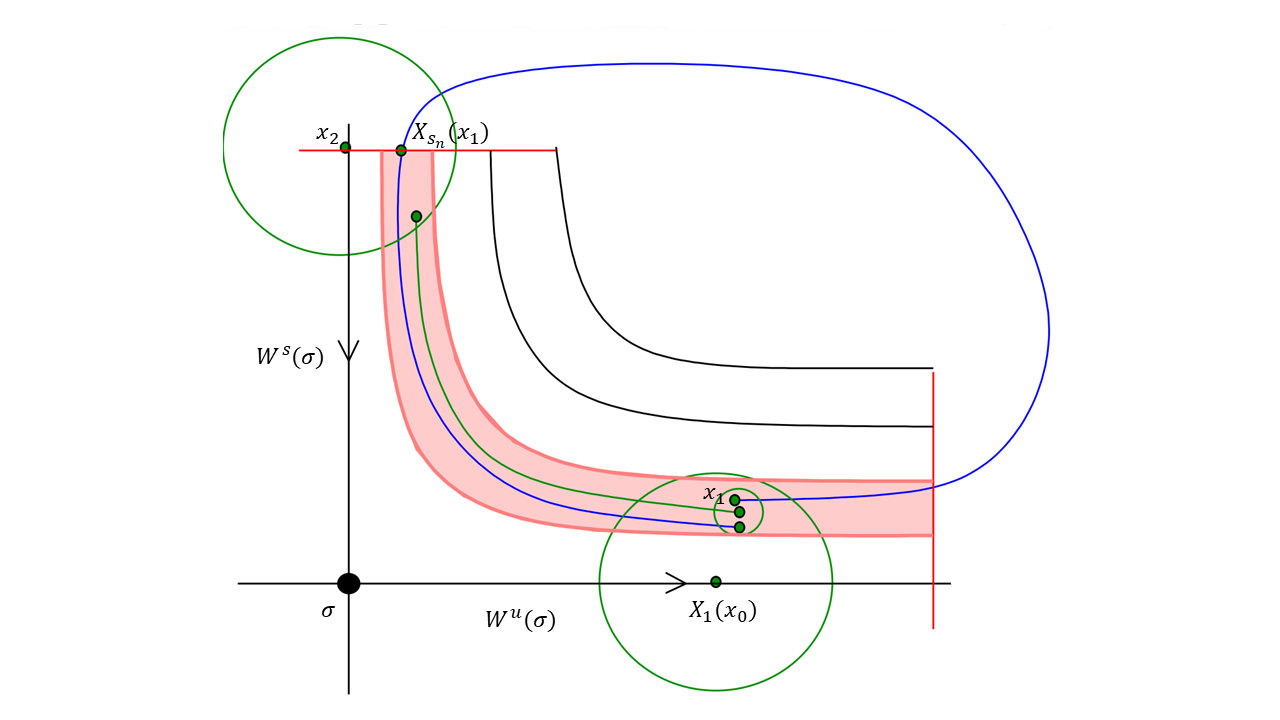}
     \caption{The $d$-$1$-pseudo-orbit $(x_i,t_i)$.}
     \label{pseudo-orbit}
 \end{figure}

The rescaled-shadowing property implies that there is some point $z\in M$ and a $h\in Rep_{\eps}$ such that $$d(X_{h(t)}(z),X_{t-s_1}(x_i))\leq \eps$$ for every $t\in \R$.
Now arguing exactly as in Lemma \ref{HL} we conclude that $z$ is a regular point contained $W^s(\sigma)\cap W^u(\sigma)$ and this finishes the proof.
\end{proof}

\begin{proof}[Proof of Theorem \ref{norshadowing}]
Let $\Lambda$ be a R-chain recurrent set and let $\sigma$ a hyperbolic singularity attached to $\Lambda$. Suppose that $X$ has the recaled-shadowing property on $\Lambda$.

\begin{enumerate}
    \item Suppose that for every adapted neighborhood of $\sigma$ there exists $N>0$ such that $n(p,V,\sigma)\leq N$. The proof that we will perform here is a slightly modification of the proof of the previous lemma. Let us proceed exactly as in the proof the Lemma \ref{HLNC} and define the $d$-$1$-pseudo orbit 

$$(x_i,t_i)=\begin{cases} (X_{i}(x_0),1), i\leq 0\\
(x_1,s_{n_0}), i=1\\
(x_2,1), i\geq2\\
\end{cases}$$ 
Let $N>0$ be such that $N\geq n(p,V,\sigma)$, for every $p\in HL(\sigma)$. Now modify the above $d$-$1$-pseudo orbit  to obtain the following $d$-$1$-pseudo orbit

$$(x_i,t_i)=\begin{cases} (X_{i}(x_0),1), i\leq 0\\
(x_i,s_{n_0}), i=1,...,N+1\\
(x_2,1), i\geq2\\
\end{cases}$$
Let $z\in \Lambda$ be an $e$-shadow for $(x_i,t_i)$. We conclude the proof by noticing that $z\in HL(\sigma)$, but $n(z,V,\sigma)\geq2N$ and this is a contradiction. 

\item Suppose that $\sigma$ has unstable index equal to $1$. Thus by following an argument similar to the argument in the proof of Theorem \ref{noshadowing}, we obtain that $\sigma$ admits at most two homoclinic loops. Therefore the result is an immediate consequence of item $(1)$.   

\item Suppose $\dim(M)\leq 3$. If $X$ has the rescaled shadowing property, then $\Lambda$ is a non-wandering  set. Now, since $\sigma$ is attached to $\Lambda$, the it must be a hyperbolic saddle. Therefore, $\sigma$ has stable or unstable index equal to $1$ and the result is an immediate consequence of item $(2)$.

\end{enumerate}

\end{proof}

\section{Further Discussion}

We end this work by discussing some  questions that remains open on the subject.
In this work we have proved the non-shadowableness of some chain-recurrent sets containing attached hyperbolic singularities.  These results are first steps towards the following general conjecture:

\begin{conj*}
Let $X$ be a $C^1$-vector field over a closed manifold. If $\Lambda\subset M$ is a chain-recurrent set with an attached singularity, then $\Lambda$ does not have the shadowing property.
\end{conj*}

This conjecture asserts that the presence of an attached singularity is enough to forbid the shadowableness of a chain-recurrent set. The following question is a general question towards the solution of the conjecture.   

\begin{question}
Are there examples of flows containing a chain-recurrent $\Lambda$ set with an attached singularity such that $\Lambda$ has the shadowing property?
\end{question}

If such an example do exist, the conjecture does not hold, but one could still ask for the non-shadowableness of flows with hyperbolic attached singularities through other techniques. This led us to the following weaker conjecture.

\begin{conj*}
Let $X$ be a $C^1$-vector field over a closed manifold. If $\Lambda\subset M$ is a chain-recurrent set with an attached hyperbolic singularity, then $\Lambda$ does not have the shadowing property.
\end{conj*}

The results developed  in this work  fully solve the previous conjecture when the ambient manifold has dimension smaller than $3$. In general dimension we solved the conjecture when the singularity under consideration is index-one.

On the other hand, our results are strongly supported on Theorem \ref{fewhomocliniloops}. Thus a natural question is if Theorem \ref{fewhomocliniloops} holds for any hyperbolic singularity.

\begin{question}
Is there an example of a flow containing a chain-recurrent $\Lambda$ set with an attached hyperbolic singularity such that the hypothesis of Theorem \ref{fewhomocliniloops} does not fit? 
\end{question}

Even one could find an example, the weaker conjecture will still remain open, since  one could try to solve it by using different techniques.

As we discussed in Section \ref{intro}, in \cite{WW} it was proved the non-shadowableness of sectional-hyperbolic chain-transitive sets with attached singularities. On the other hand, we have seen in Section \ref{prel} that there are several distinct versions of the shadowing property and they are not equivalent in presence of attached singularities. The weakest version among the ones introduced in \cite{K2} is the finite shadowing property. In \cite{ALS}, the authors proved the following result.

\begin{theorem}[\cite{ALS}]
 Let $X$ be a $C^1$-vector field  over a closed manifold $M$. If $\Lambda$ is a sectional-hyperbolic attractor with a unique boundary-type and codimension-one hyperbolic singularity, then $\Lambda$ has not the shadowing property.    
\end{theorem}

The contrast between the number of hypothesis of on the previous result and the hypothesis on the results in \cite{WW} evidences how challenging is to prove when a flow has not the finite shadowing property. Indeed, the techniques used in \cite{ALS} are totally different from the techniques used in \cite{WW}. In addition, the techniques we used in this work also do not apply to the setting of finite shadowing. This is because in all our constructions we always use infinite pseudo-obits which is forbid in the finite shadowing scenario. This lead us to the following question:

\begin{question}
can one prove that flows with a chain-recurrent set containing an attached hyperbolic singularity and without any extra tangent structure do not have the finite shadowing property? What if we assume the singularity is index-one?
\end{question}

This last question has potential to start another conjecture which is implies the two previous one.

\end{document}